\documentclass[11pt,reqno,tbtags]{article}
\usepackage{amsmath,color,amsthm}

\numberwithin{equation}{section}

\newtheorem{theorem}{Theorem}[section]
\newtheorem{corollary}[theorem]{Corollary}
\newtheorem{lemma}[theorem]{Lemma}
\newtheorem{proposition}[theorem]{Proposition}
\newtheorem{claim}[theorem]{Claim}
\newtheorem{example}[theorem]{\sl Example}

\theoremstyle{definition}
\newtheorem{remark}[theorem]{Remark}

\newcommand{\EE}{{\bf  E}}

\newcommand{\PP}{{\bf  P}}

\newcommand{\Pt}{{\tilde{P}}}

\newcommand{\Gc}{{\cal G}}
\newcommand{\Lc}{{\cal L}}

\newcommand{\xh}{\hat{x}}

\newcommand{\ph}{\hat{p}}
\newcommand{\Ph}{\widehat{P}}
\newcommand{\Gh}{\widehat{G}}
\newcommand{\Xh}{\widehat{X}}
\newcommand{\Th}{\widehat{T}}

\newcommand{\gLh}{\widehat{\gL}}

\newcommand{\begp}{\begin{proposition}}
\newcommand{\enp}{\end{proposition}}
\newcommand{\begt}{\begin{theorem}}
\newcommand{\ent}{\end{theorem}}
\newcommand{\begl}{\begin{lemma}}
\newcommand{\enl}{\end{lemma}}
\newcommand{\begc}{\begin{corollary}}
\newcommand{\enc}{\end{corollary}}
\newcommand{\begcl}{\begin{claim}}
\newcommand{\encl}{\end{claim}}
\newcommand{\begr}{\begin{remark}}
\newcommand{\enr}{\end{remark}}
\newcommand{\begal}{\begin{algorithm}}
\newcommand{\enal}{\end{algorithm}}
\newcommand{\begd}{\begin{definition}}
\newcommand{\enf}{\end{definition}}
\newcommand{\begx}{\begin{example}}
\newcommand{\enx}{\end{example}}
\newcommand{\bega}{\begin{array}}
\newcommand{\ena}{\end{array}}

\newcommand{\ignore}[1]{}

\def\rompar(#1){\textup(#1\textup)}    

\newcommand\gl{\lambda}
\newcommand\gL{\Lambda}

\newcommand\db{\overline{d}}
\newcommand\mb{\overline{m}}
\newcommand\glb{\bar{\gl}}
\newcommand\Ab{\overline{A}}

\newcommand\Pb{\overline{P}}
\newcommand\Tb{\overline{T}}

\newcommand\Xb{\overline{X}}
\newcommand\gLb{\overline{\gL}}

\newcommand{\refS}[1]{Section~\ref{#1}}
\newcommand{\refT}[1]{Theorem~\ref{#1}}

\newcommand{\refL}[1]{Lemma~\ref{#1}}
\newcommand{\refR}[1]{Remark~\ref{#1}}

\newcommand\noqed{\renewcommand{\qed}{}} 

\begin{document}

\setcounter{page}{0}
\thispagestyle{empty}

\begin{center}
{\Large \bf On hitting times and fastest strong stationary times \\ for skip-free and more general chains}
\normalsize

\vspace{4ex}
{\sc James Allen Fill\footnotemark} \\
\vspace{.1in}
Department of Applied Mathematics and Statistics \\
\vspace{.1in}
The Johns Hopkins University \\
\vspace{.1in}
{\tt jimfill@jhu.edu} and {\tt http://www.ams.jhu.edu/\~{}fill/} \\
\end{center}
\vspace{3ex}

\begin{center}
{\sl ABSTRACT} \\
\end{center}

An (upward) skip-free Markov chain with the set of nonnegative integers as state space is a chain for which upward jumps may be only of unit size; there is no restriction on downward jumps.  In a 1987 paper, Brown and Shao determined, for an irreducible continuous-time skip-free chain and any~$d$, the passage time distribution from state~$0$ to state~$d$.  When the nonzero eigenvalues $\nu_j$ of the generator on $\{0, \dots, d\}$, with~$d$ made absorbing, are all real, their result states that the passage time is distributed as the sum of~$d$ independent exponential random variables with rates $\nu_j$.  We give another proof of their theorem.  In the case of birth-and-death chains, our proof leads to an explicit representation of the passage time as a sum of independent exponential random variables.  Diaconis and Miclo recently obtained the first such representation, but our construction is much simpler.

We obtain similar (and new) results for a fastest strong stationary time~$T$ of an ergodic continuous-time skip-free chain with stochastically monotone time-reversal started in state~$0$, and we also obtain discrete-time analogs of all our results.

In the paper's final section we present extensions of our results to more general chains.

\bigskip
\bigskip

\begin{small}

\par\noindent
{\em AMS\/} 2000 {\em subject classifications.\/}  Primary 60J25;
secondary 60J35, 60J10, 60G40.
\medskip
\par\noindent
\emph{Date.} Revised May~2, 2009. 
\end{small}

\footnotetext[1]{Research supported by NSF grant DMS--0406104,
and by The Johns Hopkins University's Acheson~J.\ Duncan Fund for the
Advancement of Research in Statistics.}

\newpage

\section{Introduction and summary}
\label{S:intro}

An (upward) skip-free Markov chain on the set of nonnegative integers is a chain for which upward jumps may be only of unit size; there is no restriction on downward jumps.  Brown and Shao~\cite{BS} determined, for an irreducible continuous-time skip-free chain and any~$d$, the passage time distribution from state~$0$ to state~$d$ (we have equivalently re-identified the exponential rates):

\begin{theorem}[{\bf \cite{BS}}] \label{T:BS}
Consider an irreducible continuous-time skip-free chain~$Y$ on the nonnegative integers with $Y(0) = 0$.  Given~$d$, let $X$ (with state space $\{0, \dots, d\}$) be obtained from~$Y$ by making~$d$ an absorbing state, and let~$G$ denote the generator for~$X$.  Then the hitting time of state~$d$ (same for~$X$ and~$Y$)  has Laplace transform
$$
u \mapsto \prod_{j = 0}^{d - 1} \frac{\nu_j}{\nu_j + u},
$$
where $\nu_0, \dots, \nu_{d - 1}$ are the~$d$ nonzero eigenvalues of $- G$ (known to have positive real parts).  In particular, if the $\nu_j$'s are real, then the hitting time distribution is the convolution of \emph{Exponential$(\nu_j)$} distributions.  
\end{theorem}

In this paper we give a new proof for \refT{T:BS}.  In the case of birth-and-death chains (which are time-reversible and thus have real eigenvalues), the result is often attributed to Keilson but in fact (consult~\cite{DM}) dates back at least to Karlin and McGregor~\cite{KM}.  Our proof leads in this case to an explicit (sample-path) representation of the passage time as a sum of independent exponential random variables.   Diaconis and Miclo~\cite{DM} recently obtained the first such representation for birth-and-death chains, but our construction is much simpler.
\smallskip
 
There is an analog for discrete time:

\begin{theorem}\label{T:BSdisc}
Consider an irreducible discrete-time skip-free chain~$Y$ on the nonnegative integers with $Y(0) = 0$.  Given~$d$, let $X$ (with state space $\{0, \dots, d\}$) be obtained from~$Y$ by making~$d$ an absorbing state, and let~$P$ denote the transition matrix for~$X$.  Then the hitting time of state~$d$ (same for~$X$ and~$Y$)  has probability generating function
$$
u \mapsto \prod_{j = 0}^{d - 1} \left[ \frac{(1 - \theta_j) u}{1 - \theta_j u} \right],
$$
where $\theta_0, \dots, \theta_{d - 1}$ are the~$d$ non-unit eigenvalues of~$P$.  In particular, if every $\theta_j$ is real and nonnegative, then the hitting time distribution is the convolution of \emph{Geometric$(1 - \theta_j)$} distributions. 
\end{theorem}

\begin{remark}
The conclusions of Theorems~\ref{T:BS} and~\ref{T:BSdisc} hold for all skip-free chains~$X$ on $\{0, \dots, d\}$ for which $g_{i, i + 1} > 0$ (respectively, $p_{i, i + 1} > 0$) for $i = 0, \dots, d - 1$.  Indeed, our proofs extend to that case, or one can extend the theorems by a simple perturbation argument. 
\end{remark}
 
Results similar to Theorems~\ref{T:BS}--\ref{T:BSdisc} can also be established for fastest strong stationary times (consult~\cite{DF}, \cite{TTS}, and~\cite{dualityc} for general background, and~(3.3) and~(3.5) of~\cite{dualityc} for how to check from the generator~$G$ of a continuous-time chain whether the chain has the ``monotone likelihood ratio'' property, that is, whether its time-reversal is stochastically monotone):

\begin{theorem} \label{T:TTS} 
\ \ \\
\emph{(a)} Consider an ergodic (equivalently, irreducible) continuous-time skip-free chain~$X$ on the state space $\{0, \dots, d\}$ with $X(0) = 0$ and stochastically monotone time-reversal.  Let~$G$ denote the generator for~$X$.  Then a fastest strong stationary time (SST) for~$X$ has Laplace transform
$$
u \mapsto \prod_{j = 0}^{d - 1} \frac{\nu_j}{\nu_j + u},
$$
where $\nu_0, \dots, \nu_{d - 1}$ are the~$d$ nonzero eigenvalues of $- G$ (known to have positive real parts).  In particular, if the $\nu_j$'s are real, then the fastest SST distribution is the convolution of \emph{Exponential$(\nu_j)$} distributions. 
\smallskip

\noindent
\emph{(b)} Consider an ergodic (equivalently, irreducible and aperiodic) discrete-time skip-free chain~$X$ on the state space $\{0, \dots, d\}$ with $X(0) = 0$ and stochastically monotone time-reversal.  Let~$P$ denote the transition matrix for~$X$.  Then a fastest strong stationary time (SST) for~$X$ has probability generating function
$$
u \mapsto \prod_{j = 0}^{d - 1} \left[ \frac{(1 - \theta_j) u}{1 - \theta_j u} \right],
$$
where $\theta_0, \dots, \theta_{d - 1}$ are the~$d$ non-unit eigenvalues of~$P$.  In particular, if every $\theta_j$ is real and nonnegative, then the fastest SST distribution is the convolution of \emph{Geometric$(1~-~\theta_j)$} distributions. 
\end{theorem}

Our proofs of Theorems~\ref{T:BS}--\ref{T:BSdisc} and~\ref{T:TTS} could hardly be simpler in concept.  Take \refT{T:BSdisc}, for instance.  Let~$\theta_0, \dots, \theta_{d - 1}$ be ordered arbitrarily.  We will exhibit a matrix~$\gL$, with rows and columns indexed by the state space $\{0, \dots, d\}$, having the following properties:
\begin{enumerate}
\item $\gL$ is lower triangular.
\item The rows of~$\gL$ sum to unity.
\item $\gL P = \Ph \gL$, where $\Ph$ is defined by
\begin{equation}
\label{phat}
\ph_{i j} := 
 \begin{cases}
 \theta_i   & \mbox{if $j = i$} \\
 1 - \theta_i   & \mbox{if $j = i + 1$} \\
 0   & \mbox{otherwise},
 \end{cases} 
\end{equation}
for $i, j = 0, \dots, d$ (with $\theta_d = 1$).
\end{enumerate}
We will prove that from these three properties follows
\begin{equation}
\label{key}
\PP(T \leq t) = \Ph^t(0, d), \quad t = 0, 1, \dots,
\end{equation}
where~$T$ is the absorption time in question.  [In fact, all we will need to prove~\eqref{key} are~(iii) and the following two easy consequences of (i)--(ii):
\begin{enumerate}
\item[(iv)] $\gL(0, k) = \delta_{0, k}$.
\item[(v)] $\gL(k, d) = 0$ if $k \neq d$.]
\end{enumerate}
In the special case that the $\theta_j$'s are real and nonnegative, then \refT{T:BSdisc} follows immediately from~\eqref{key}; the general case will not take too much additional work.  Moreover, for birth-and-death chains with nonnegative eigenvalues the proof will reveal how to identify (i.e.,\ construct) geometric random variables summing to~$T$.

In Section~1 of~\cite{FillKeil} it was shown how that paper's main continuous-time theorem (Theorem~1.1) follows immediately from the main discrete-time theorem (Theorem~1.2).  By the same one-line proof, in the present paper \refT{T:BS} follows immediately from \refT{T:BSdisc}.  For this reason we choose not to present a direct proof of \refT{T:BS}, but in the case of birth-and-death chains we will give the explicit construction of independent exponential random variables summing to the absorption time.  For the same reason we will prove only part~(b) of \refT{T:TTS}.

Here is an outline for the paper.  In \refS{S:12proof} we prove \refT{T:BSdisc}, and in \refS{S:constructs} we give stochastic constructions for Theorems~\ref{T:BS}--\ref{T:BSdisc} in the special case of birth-and-death chains.  \refS{S:TTSproof} is devoted to a proof of \refT{T:TTS}(b).
In \refS{S:general} we present extensions of our results, including stochastic constructions, to more general chains.

\section{Proof of \refT{T:BSdisc}} \label{S:12proof}

In this section we prove~\refT{T:BSdisc}, following the outline near the end of \refS{S:intro}.  Denote the eigenvalues of~$P$, with each distinct eigenvalue listed as many times as its algebraic multiplicity, by $\theta_0, \dots, \theta_d$.  We claim that precisely one of these, say $\theta_d$, equals~$1$ and that the rest have modulus smaller than~$1$; how the rest are ordered will for the majority of our results be irrelevant (we shall take special notice otherwise).  

Here is a proof of the simple claim.  Let $P'$ denote the leading principal $d$-by-$d$ submatrix of~$P$.  Expanding the determinant of $P - \theta I$ on the last row, we find that the $d + 1$ eigenvalues of~$P$ are $\theta = 1$ together with the~$d$ eigenvalues of~$P'$.  But by our assumptions that $p_{i, i + 1} > 0$ for $t = 0, \dots, d - 1$ and~$d$ is absorbing, we have $(P')^t \to 0$ and so (see for example \cite[Theorem 5.6.12]{HJ}) $\rho(P') < 1$.

Let~$I$ denote the identity matrix and define
\begin{equation}
\label{Qkdef}
Q_k := (1 - \theta_0)^{-1} \cdots (1 - \theta_{k-1})^{-1} (P - \theta_0 I) \cdots (P - \theta_{k-1} I), \quad k = 0, \dots, d,
\end{equation}
with the natural convention $Q_0 := I$.  Note that
\begin{equation}
\label{Qevo}
Q_k P = \theta_k Q_k + (1 - \theta_k) Q_{k + 1}, \quad  k = 0, \dots, d - 1,
\end{equation}
and that the rows of each~$Q_k$ all sum to~$1$.

Define
\begin{equation}
\label{gLdef}
\gL(i, j) := Q_i(0, j), \quad i, j = 0, \dots, d.
\end{equation}

\begin{lemma} \label{L:gL}
Let~$P$ be as in \refT{T:BSdisc}.  Then the matrix~$\gL$ defined at~\eqref{gLdef} enjoys properties \emph{(i)--(iii)} listed in \refS{S:intro}.
\end{lemma}

\begin{proof}
\noqed
The rows of~$\gL$ sum to unity because each of the basic factors $(1 - \theta_r)^{-1} (P - \theta_r I)$ in~\eqref{Qkdef} has that property; and since~$P$ is skip-free, $\gL$ is lower triangular.  Our next claim is that $\gL P = \Ph \gL$, where~$\Ph$ is defined at~\eqref{phat}.  Indeed, equality of $k$th rows is clear from~\eqref{Qevo} for $k = 0, \dots, d - 1$ and from the Cayley--Hamilton theorem for $k = d$.
\end{proof}

\begin{remark} \label{R:nonsingular}
It is not surprising that the eigenvalues of~$P$ line the diagonal of the upper triangular 
matrix~$\Ph$, because our~$\gL$ is nonsingular and so property~(iii) implies that~$P$ 
and~$\Ph$ are similar.  To see that the lower triangular matrix $\gL$ is nonsingular, we check that its diagonal entries are all nonzero.  Indeed, for $k = 0, \dots, d$, we use the skip-free property 
of~$P$ to calculate:
$$
\gL(k, k) = Q_k(0, k) = \frac{p_{0, 1} \cdots p_{k - 1, k}}{(1 - \theta_0) \cdots (1 - \theta_{k - 1})} \neq 0. 
$$
\end{remark}
\medskip

So we have found a matrix~$\gL$ satisfying properties (i)--(iii), and properties (iv)--(v) follow easily from (i)--(ii).  Notice that property~(iii) immediately extends to
\begin{equation}
\label{twine}
\gL P^t = \Ph^t \gL, \quad t = 0, 1, \dots.
\end{equation}

\begin{lemma} \label{L:key}
Let~$P$ and~$X$, with absorption time~$T$, be as in \refT{T:BSdisc} and let~$\Ph$ be the bidiagonal matrix defined at~\eqref{phat}.  Then~\eqref{key} holds:
$$
\PP(T \leq t) = \Ph^t(0, d), \quad t = 0, 1, \dots.
$$
\end{lemma} 

\begin{proof}
Let~$\gL$ be any matrix enjoying properties (iii)--(v) of \refS{S:intro}; recall that Lemma \ref{L:gL} provides such a~$\gL$.  We will use~$\gL$ to establish the desired result, and the consequence~\eqref{twine} of property~(iii) will be key.

We consider the $(0, d)$ entry of each side of~\eqref{twine}.  On the left we have
$$
(\gL P^t)(0, d) = \sum_k \gL(0, k) P^t(k, d) = P^t(0, d) = \PP(T \leq t),
$$
where the second equality is immediate from property~(iv).  On the right we have
$$
(\Ph^t \gL)(0, d) = \sum_k \Ph^t(0, k) \gL(k, d) = \Ph^t(0, d) \gL(d, d),
$$
where the second equality is immediate from property~(v).  So all that remains is to establish that $\gL(d, d) = 1$.

We do this by passing to the limit as $t \to \infty$ in our now-established equation
$$
\PP(T \leq t) = \Ph^t(0, d) \gL(d, d).
$$
The limit on the left is, of course, $1$.  Moreover, we claim that $\Ph^t(0, d) \to 1$, and then the proof of the lemma will be complete.  The claim is probabilistically obvious when the eigenvalues are all real and nonnegative.  To see the claim in general, write $\Ph$ as the bordered matrix
$$
\Ph =
\left[ 
\begin{array}{cc}
A & b \\
0 & 1
\end{array}
\right]
$$  
by breaking off the last row and column.  The matrix~$A$ has spectral radius $\rho(A) < 1$; and  because the rows of~$\Ph$ sum to unity, we have $b = (I - A) {\bf 1}$, where~${\bf 1}$ denotes the vector of~$d$ ones.  In this notation we find
$$
\Ph^t =
\left[ 
\begin{array}{cc}
A^t & (I + A + \cdots + A^{t - 1}) b \\
0 & 1
\end{array}
\right]
\to
\left[
\begin{array}{cc}
0 & (I - A)^{-1} b \\
0 & 1
\end{array}
\right],
$$
which is a matrix with~$1$s throughout the last column and~$0$s elsewhere.  In particular, 
$\Ph^t(0, d) \to 1$. 
\end{proof}

As mentioned in \refS{S:intro}, \refL{L:key} is clearly all that is needed to prove \refT{T:BSdisc} when the eigenvalues of~$P$ are nonnegative real numbers.  In general we proceed as follows.  From \refL{L:key} it follows that~$T$ has probability generating function
\begin{equation}
\label{pgf}
\EE\,u^T = (1 - u)\,(I - u \Ph)^{-1} (0, d), \quad |u| < 1.
\end{equation}
But the simple form of~$\Ph$ makes it an easy matter to invert $1 - u \Ph$ explicitly:\ the inverse is upper triangular, with
\begin{equation}
\label{inverse}
(I - u \Ph)^{-1} (i, j) = \frac{(1 - \theta_i) \cdots (1 - \theta_{j - 1}) u^{j - i}}{(1 - \theta_i u) \cdots (1 - \theta_j u)}, \qquad 0 \leq i \leq j \leq d.
\end{equation}
Taking $i = 0$ and $j = d$ and combining \eqref{pgf}--\eqref{inverse} gives~\refT{T:BSdisc}.
\ignore{
When $\theta_0, \dots, \theta_{d - 1}$ are distinct, we have the explicit expression
\begin{equation}
\label{explicit}
\Ph^t(0, d) = 1 - \sum_{k = 0}^{d - 1} \theta^t_k \prod \frac{1 - \theta_r}{\theta_k - \theta_r},
\end{equation}
where the product is over $r \neq k$ satisfying $0 \leq r \leq d - 1$.
}

\section{Stochastic constructions for Theorems~\ref{T:BS} and~\ref{T:BSdisc}} \label{S:constructs}

In this section we consider Theorems~\ref{T:BS} and~\ref{T:BSdisc} again in the case of a birth-and-death chain.  For such a chain, the eigenvalues $\nu_j$ or $\theta_j$ are all real; this can be seen by perturbing (by an arbitrarily small amount) to get an ergodic generator or kernel, which is time-reversible and thus diagonally similar to a symmetric matrix.  We show how the proof in \refS{S:12proof} of \refT{T:BSdisc} yields an explicit construction of independent geometric random variables summing to the hitting time in the case of a birth-and-death chain having nonnegative eigenvalues (for which it is sufficient that the holding probabilities satisfy $p_{i i} \geq 1 / 2$ for all~$i$).  We also exhibit an analogous representation of the hitting time in \refT{T:BS} as a sum of independent exponential random variables in the case of a continuous-time birth-and-death chain.

We begin with the discrete case.  Returning to~\eqref{Qkdef}, we now suppose that~$P$ is a discrete-time birth-and-death kernel with nonnegative eigenvalues $\theta_j$ ordered so that $0 \leq \theta_0 \leq \cdots \leq \theta_{d - 1} < \theta_d = 1$.  The polynomials $(P - \theta_0 I) \cdots (P - \theta_{k - 1}I)$ in~$P$ appearing at~\eqref{Qkdef} are called \emph{spectral polynomials}.  We claim that the spectral polynomials are all nonnegative matrices, that is, that the matrices~$Q_k$ at~\eqref{Qkdef} are stochastic; this follows immediately by perturbation from Lemma~4.1 in~\cite{FillKeil} (which concerns ergodic birth-and-death kernels and is in turn an immediate consequence of \cite[Theorem~3.2]{MW}.  Reversibility of~$P$ plays a key role.)  Thus the matrix~$\gL$ defined at~\eqref{gLdef} is stochastic, as of course is the pure-birth kernel~$\Ph$ defined at~\eqref{phat}.  For stochastic~$P$, $\Ph$, and~$\gL$, the identity $\gL P^t = \Ph^t \gL$ at~\eqref{twine} is read as ``the semigroups $(P^t)_{t \geq 0}$ and $(\Ph^t)_{t \geq 0}$ are intertwined by the link~$\gL$''.

Whenever we have such an intertwining and (to be specific) $X_0 = 0$, Section~2.4 of~\cite{DF} shows one way to construct explicitly, from~$X$ and independent randomness, another Markov chain, say~$\Xh$, with kernel~$\Ph$ such that
\begin{equation}
\label{history}
\Lc(X_t\,|\,\Xh_0, \dots, \Xh_t) = \gL(\Xh_t, \cdot) \quad \mbox{for all~$t$}.
\end{equation}
In particular, since the link~$\gL$ is lower triangular [by \refL{L:gL}(i)] and $\gL(d, d) = 1$ (recall the proof of \refL{L:key}), so that $\gL(d, \cdot)$ is unit mass at~$d$, it follows that~$\Xh$ and~$X$ (and also~$Y$ of \refT{T:BSdisc}) have the same absorption time~$T$.  The discussion in~\cite{DF} is used in that paper in the case that~$P$ is ergodic and~$\Ph$ is absorbing, but it applies equally well to any intertwining.  Once we have built the pure birth chain $\Xh$, the independent geometric random variables summing to~$T$ are simply the waiting times between successive births in~$\Xh$. 

Here is our construction of~$\Xh$; it has the same form (but with~$\gLh$ changed to~$\gL$) as in Section~4.1 of~\cite{FillKeil}, which applied to a different setting, and we repeat it here for the reader's convenience.  The chain~$X$ starts with $X_0 = 0$ and we set $\Xh_0 = 0$.  Inductively, we will have $\gL(\Xh_t, X_t) > 0$ (and so $X_t \leq \Xh_t$) at all times~$t$.  The value we construct for $\Xh_t$ depends only on the values of~$\Xh_{t - 1}$ and~$X_t$ and independent randomness.  Indeed, given $\Xh_{t - 1} = \xh$ and $X_t = y$, if $y \leq \xh$ then our construction sets $\Xh_t = \xh + 1$ with probability
\begin{equation}
\label{phatpath}
\frac{\Ph(\xh, \xh + 1) \gL(\xh + 1, y)}{(\Ph \gL)(\xh, y)} = \frac{(1 - \theta_{\xh}) \gL(\xh + 1, y)}{\theta_{\xh} \gL(\xh, y) + (1 - \theta_{\xh}) \gL(\xh + 1, y)} =  \frac{(1 - \theta_{\xh}) Q_{\xh + 1}(0, y)}{(Q_{\xh} P)(0, y)}
\end{equation}
and $\Xh_t = \xh$ with the complementary probability; if $y = \xh + 1$ (which is the only other possibility, since $y = X_t \leq X_{t - 1} + 1 \leq \xh + 1$ by induction), then we set $\Xh_t = \xh + 1$ with certainty.

\begin{remark}
(a) By~\eqref{history} and the lower-triangularity of~$\gL$, our construction satisfies $X_t \leq \Xh_t$ for all~$t$.   Thus, among all discrete-time birth-and-death chains on $\{0, \dots, d\}$ started at~$0$ and with absorbing state~$d$ and given nonnegative eigenvalues $0 \leq \theta_0 \leq \theta_1 \leq \cdots \leq \theta_{d - 1} < \theta_d = 1$, the pure-birth ``spectral'' kernel~$\Ph$ defined at~\eqref{phat} is stochastically maximal at every epoch~$t$.

(b) In the general setting of~\refT{T:BSdisc}, we do not know any broad class of examples other than the birth-and-death chains we have just treated for which the eigenvalues are nonnegative real numbers and the spectral polynomials are nonnegative matrices.  Nevertheless, the stochastic construction of the preceding paragraph applies verbatim to all such chains.
\end{remark}

For continuous-time birth-and-death chains, and more generally for continuous-time skip-free chains with real eigenvalues and generator~$G$ having nonnegative spectral polynomials
$$
(G + \nu_0 I) \cdots (G + \nu_{k - 1} I), \quad k = 0, \dots, d,
$$
there is an analogous construction.  It takes a bit more effort to describe than for discrete-time chains, so we refer the reader to Section~5.1 of~\cite{FillKeil}, where again~$\gLh$ there is changed to~$\gL$ here, for details.  Briefly, if the bivariate chain $(\Xh, X)$ is in state $(\xh, x)$ at a given jump time, then we construct an exponential random variable with rate 
$$
r = \nu_{\xh} \gL(\xh + 1, x) / \gL(\xh, x).
$$
If~$X$ jumps before the lifetime of this exponential expires, then~$\Xh$ holds unless~$X$ jumps to $\xh + 1$, in which case $\Xh$ also jumps to $\xh + 1$.  But if the exponential expires first, then at that expiration time~$X$ holds but~$\Xh$ jumps to $\xh + 1$.

\begin{remark}
\label{R:stochmaxcont}
Among all continuous-time birth-and-death chains on $\{0, \dots, d\}$ started at~$0$ and with absorbing state~$d$ and given eigenvalues $\nu_0 \geq \nu_1 \geq \cdots \geq \nu_{d - 1} > \nu_d = 0$ for $- G$, the pure-birth ``spectral'' chain with birth rate $\nu_i$ at state~$i$ for $i = 0, \dots, d$ is stochastically maximal at every epoch~$t$.
\end{remark}

\section{Proof of \refT{T:TTS}(b)} \label{S:TTSproof}

In this section we prove~\refT{T:TTS}(b).  The proof is quite similar to that of \refT{T:BSdisc}, so we shall be brief.  Denote the eigenvalues of~$P$ by $\theta_0, \dots, \theta_d$.   Since we assume that~$P$ is ergodic, precisely one of these, say $\theta_d$, equals~$1$ and the rest have modulus smaller than~$1$.

Define the matrices~$Q_k$ as at~\eqref{Qkdef} [so that~\eqref{Qevo} holds] and define~$\gL$ as at~\eqref{gLdef}.  Then \refL{L:gL} holds for~$P$ of \refT{T:TTS}(b) by the same proof, as do properties (iv)--(v) of \refS{S:intro} and the 
relation~\eqref{twine}.  As we will show next, \refL{L:key} also holds in the present setting, and then the same argument as in the final paragraph of \refS{S:12proof} completes the proof of \refT{T:TTS}(b). 

\begin{lemma} \label{L:TTSkey}
Let~$P$ and~$X$, with fastest strong stationary time~$T$, be as in \refT{T:TTS}(b) and let~$\Ph$ be the bidiagonal matrix defined at~\eqref{phat}, where $\theta_0, \dots, \theta_{d - 1}$ are the non-unit eigenvalues of~$P$.  Then
$$
\PP(T \leq t) = \Ph^t(0, d), \quad t = 0, 1, \dots.
$$
\end{lemma} 

\begin{proof}
As in the proof of \refL{L:key}, the key is the intertwining relation
\begin{equation}
\label{twine2} 
\gL P^t = \Ph^t \gL, \quad t = 0, 1, \dots,
\end{equation}
together with properties (iv)--(v) of \refS{S:intro}.  We consider the $(0, d)$ entries in~\eqref{twine2}.  

On the left we have
\begin{equation}
\label{lefttwine}
(\gL P^t)(0, d) = \sum_k \gL(0, k) P^t(k, d) = P^t(0, d) = \PP(T \leq t)\,\pi(d),
\end{equation}
where~$\pi$ is the stationary distribution for~$P$.  We explain how to obtain the final equality.  By the assumption that the time-reversal (call it~$\Pt$) of~$P$ is stochastically monotone, the ratio $P^t(0, x) / \pi(x) = \Pt^t(x, 0) / \pi(0)$ is minimized by $x = d$, with minimum value (by definition) $1 - s(t)$, where $s(t)$ is the so-called separation from stationarity for~$X$.  But, as is well known, a fastest strong stationary time~$T$ for~$X$ satisfies
$$
\PP(T \leq t) = 1 - s(t) \quad \mbox{for all~$t$}.
$$

On the right at~\eqref{twine2} we have
\begin{equation}
\label{righttwine}
(\Ph^t \gL)(0, d) = \sum_k \Ph^t(0, k) \gL(k, d) = \Ph^t(0, d) \gL(d, d).
\end{equation}
Now equate~\eqref{lefttwine} and~\eqref{righttwine} and pass to the limit as $t \to \infty$.  On the left, $\PP(T \leq t) \to 1$ (by ergodicity of~$X$); on the right, $\Ph^t(0, d) \to 1$ by the argument presented in the final paragraph of the proof of~\refL{L:key}.  So $\gL(d, d) = \pi(d)$, and now once again equating~\eqref{lefttwine} and~\eqref{righttwine} and canceling the factor $\pi(d) > 0$ gives the desired result.

\end{proof}

Stochastic constructions for \refT{T:TTS} have already been given for birth-and-death chains, in discrete time with nonnegative eigenvalues and in continuous time, in Section~4.1 (respectively, Section~5.1) in~\cite{FillKeil}.  The constructions extend verbatim to all skip-free chains with real (and, in discrete time, nonnegative) eigenvalues and nonnegative spectral polynomials.

\section{General chains} \label{S:general}

Our proof in \refS{S:12proof} of the central \refT{T:BSdisc} for the absorption time of a 
discrete-time skip-free chain rested on the construction of a matrix~$\gL$ having the properties (i)--(iii) listed in \refS{S:intro}.  The question we wish to address in this section are:
\begin{enumerate}
\item[(a)] Can \refT{T:BSdisc} be extended to general absorbing chains?
\item[(b)] Is the spectral-polynomials construction of $\gL$ inevitable?  That is, is the matrix $\gL$ uniquely determined by the properties (i)--(iii) listed in \refS{S:intro}?
\item[(c)] If the eigenvalues and spectral polynomials of a general absorbing chain are all nonnegative, can the stochastic construction for \refT{T:BSdisc} be extended? 
\end{enumerate}
As we shall see, the answer to each of these questions is {\bf yes}.  Question~(b) arises naturally because (switching to continuous time and limiting attention to birth-and-death chains for the moment, since that is the setting of~\cite{DM}) the proofs of \refT{T:BS} both of Diaconis and Miclo~\cite{DM} and of the present paper rely on construction of a link~$\gL$ such that $\gL G = \Gh \gL$, where $\Gh$ (the analogue of $\Ph$ in continuous time) is the pure-birth ``spectral'' generator described in our \refR{R:stochmaxcont}.  The two methods of construction are strikingly different, so it is interesting that the end-product $\gL$ is the same.

All three questions can be addressed in the following setting generalizing that of \refT{T:BSdisc}.  (We shall subsequently call this ``the general setting''.)  Consider a discrete-time Markov chain~$X$ with state space $\{0, \dots, d\}$ and arbitrary initial distribution $m_0$ (regarded as a row vector in later calculations) and transition matrix~$P$.  Assume that state~$d$ is absorbing and accessible from each other state.  
(There is no loss of generality in restricting the absorbing set to be a singleton.)  
Let $\theta_0, \dots, \theta_{d - 1}$ be the~$d$ non-unit eigenvalues of~$P$ (in fixed but arbitrary order---except where we require the nonnegativity of spectral polynomials, in which case we use nondecreasing order), let 
$\theta_d = 1$, and let~$\Ph$ be defined by~\eqref{phat}.  All square matrices we consider have rows and columns indexed by the state space $\{0, \dots, d\}$.  Let $Q_0, \dots, Q_d$ be the normalized spectral polynomials defined at~\eqref{Qkdef}.

Questions~(a) and~(b) are answered in \refS{SS:absgen} (see \refL{L:gLgen} and \refT{T:gen} respectively); question~(c), in \refS{SS:modlink}.

\subsection{Absorption times for general chains and the inevitability of spectral polynomials} 
\label{SS:absgen}

Our first result of this subsection demonstrates the inevitable use of spectral polynomials in the construction of~$\gL$, even in the present general setting, provided $\lambda_0 = m_0$.

\begl
\label{L:gLgen}
The unique matrix~$\gL$ {\rm (}with rows denoted by $\gl_0, \dots, \gl_d${\rm )} satisfying the two conditions
\begin{equation}
\label{desired}
m_0 = \gl_0 \qquad\mathrm{and}\qquad \gL P = \Ph \gL
\end{equation}
is given by
\begin{equation}
\label{forced}
\gl_i = m_0 Q_i, \quad i = 0, \dots, d.
\end{equation}
\enl

\begin{proof}
It is easy to check, as in the proof of \refL{L:gL}, that the choice~\eqref{forced} satisfies~\eqref{desired}.  Conversely, the $i$th row of $\gL P = \Ph \gL$ ($i = 0, \dots, d - 1$) requires
$$
\gl_i P = \theta_i \gl_i + (1 - \theta_i) \gl_{i + 1},\quad\mathrm{i.e.,}\quad \gl_{i + 1} = \gl_i \left[ (1 - \theta_i)^{-1} (P - \theta_i I) \right],
$$
and so (by induction) \eqref{desired} implies~\eqref{forced}.
\end{proof}

Our next result (\refT{T:gen}), the main result of this section, greatly generalizes \refT{T:BSdisc} by providing a tidy representation of the absorption-time distribution in the general setting.  For the purposes of this result, we use the conventions that an empty sum vanishes and that 
$\gL(-1, d) := 0$ and $\gL(d + 1, d) := 1$.  We also use the notation
\begin{equation}
\label{akdef}
a_k := \gL(k, d) - \gL(k - 1, d), \quad k = 0, \dots, d + 1. 
\end{equation}
Observe that these $a_k$'s sum to unity, and in the proof of \refT{T:gen} we show that $a_{d + 1}$ always vanishes.  Moreover, if the spectral polynomials happen to be nonnegative, it is easily verified that ${\bf a} := (a_0, \dots, a_{d + 1})$ has real and nonnegative entries and so is a probability mass function.  Finally, if the eigenvalues $\theta_i$ are all real and nonnegative, then for $k = 0, \dots, d$ we let $\Gc(\theta_0, \dots, \theta_{k - 1})$ denote the convolution of geometric distributions with respective success probabilities $1 - \theta_0, \dots, 1 - \theta_{k - 1}$; we utilize the natural conventions here that this distribution is concentrated at~$0$ when $k = 0$. 

\begt
\label{T:gen}
In the above general setting and notation, in particular with~$\Ph$ defined at~\eqref{phat}, $\gL$ defined as in \refL{L:gLgen} at~\eqref{forced}, and $a_k$ defined at~\eqref{akdef}, the absorption time~$T$ satisfies
\begin{equation}
\label{cdf}
\PP(T \leq t) = \sum_{k = 0}^d a_k \sum_{j = k}^d \Ph^t(0, j), \quad t = 0, 1, 2, \dots,
\end{equation}
with probability generating function
\begin{equation}
\label{pgfgen}
\EE\,u^T = \sum_{k = 0}^d a_k \prod_{j = 0}^{k - 1} \left[ \frac{(1 - \theta_j) u}{1 - \theta_j u} \right].
\end{equation}
In particular, if the eigenvalues $\theta_i$ are all nonnegative real numbers and the spectral polynomials in~$P$ are all nonnegative matrices, then~$T$ is distributed as the ${\bf a}$-mixture $\sum_{k = 0}^d a_k\,\Gc(\theta_0, \dots, \theta_{k - 1})$ of the convolution distributions $\Gc(\theta_0, \dots, \theta_{k - 1})$, $k = 0, \dots, d$.
\ent

\begin{proof}
\refL{L:gLgen} is the key.  Readily generalizing the proof of \refL{L:key} and then using summation by parts, one finds
$$
\PP(T \leq t) = \sum_{j = 0}^d \Ph^t(0, j) \gL(j, d) = \sum_{k = 0}^d a_k \sum_{j = k}^d 
\Ph^t(0, j).
$$
As shown in the proof of \refL{L:key}, $\Ph^t(0, j) \to \delta_{d,\,j}$ as $t \to \infty$; thus $1 = \sum_{k = 0}^d a_k$ and so $a_{d + 1} = 0$. 
Equation~\eqref{cdf} is all that is needed to establish~\eqref{pgfgen} when the eigenvalues of~$P$ are nonnegative real numbers; in general one can use~\eqref{inverse} (we omit the routine details).
\end{proof}

\begr
\label{R:gen}
(a)~If the chain is upward skip-free and $m_0 = \delta_0$, then $a_k \equiv \delta_{d, k}$ and \refT{T:BSdisc} is recovered.

(b)~The following special case is treated in detail by Miclo~\cite{Miclo}, albeit in somewhat different fashion.  If there exists~$\pi$ satisfying the detailed balance condition 
$\pi_i p_{i j} = \pi_j p_{j i}$ for all $i, j \in \{0, \dots, d - 1\}$, then by an argument like the one at the beginning of our \refS{S:constructs} the eigenvalues of~$P$ are nonnegative reals and the spectral polynomials are nonnegative.  Thus, by \refT{T:gen}, the absorption time is distributed as a mixture of convolutions of geometric distributions.  Miclo also shows that $a_d > 0$ when the states in $\{0, \dots, d - 1\}$ all communicate with one another. 

(c)~At least one case other than that of part~(b) of this remark is known for which the eigenvalues and spectral polynomials are all nonnegative.  If~$P$ is upper triangular, then of course the eigenvalues are nonnegative reals, and He and Zhang~\cite[Appendix~A]{HZ} prove that the spectral polynomials are nonnegative.

(d)~A result analogous to \refT{T:gen} holds for continuous-time chains; we omit the details.  See Miclo~\cite[Section~1]{Miclo} for a discussion of connections with the extensive literature on so-called ``phase-type'' distributions.

(e)~A result similar to \refT{T:gen} (with a similar proof) holds for the distribution of the fastest strong stationary time of a general ergodic chain~$X$ with general initial distribution~$m_0$ and stationary distribution~$\pi$, provided $\PP(X_t = x) / \pi(x)$ is minimized for every~$t$ by the choice $x = d$.  A sufficient condition for this is that the time-reversal of~$P$ is stochastically monotone (with respect to the natural linear order on $\{0, \dots, d\}$) and $m_0(x) / \pi(x)$ is decreasing in~$x$ (for example, $m_0 = \delta_0$).  The theorem then has the same form as \refT{T:gen}, except that now $a_k$ needs to be defined as 
$$
[\gL(k, d) - \gl(k - 1, d)] / \pi(d).
$$ 
\enr

\subsection{Stochastic construction via a modified link}
\label{SS:modlink}

We continue to study the general setting.  Whenever the eigenvalues and spectral polynomials are all nonnegative, the link~$\gL$ of~\eqref{forced} provides an intertwining of the semigroups 
$(P^t)_{t \geq 0}$ and $(\Ph^t)_{t \geq 0}$ [recall~\eqref{desired}], and again (as in \refS{S:constructs}) a chain~$\Xh$ with kernel~$\Ph$ can be constructed such that $\Xh_0 = 0$ and the ``sample-path intertwining'' relation~\eqref{history} holds.  However, unlike for skip-free chains, in the general case although we do have $\gL(d, d) = 1$ (this simply restates our earlier observation that $a_{d + 1} = 0$) there is no guarantee that the link~$\gL$ is lower triangular and thus all we can say with certainty is that the absorption times~$T$ for~$X$ and~$\Th$ (say) for~$\Xh$ satisfy $T \leq \Th$.  We will rectify this situation by modifying the link~$\gL$; this will not contradict the uniqueness of~$\gL$ proven in \refL{L:gLgen}, because we will also substitute a different ``dual'' kernel $\Pb$ for the spectral kernel $\Ph$.

To set the stage in the general setting without yet imposing any assumptions about nonnegativity of eigenvalues or spectral polynomials, define the matrix~$\gLb$ with rows $\glb_0, \dots, \glb_d$ by setting
$$
\glb_i :=
\begin{cases}
[1 - \gl_i(d)]^{-1} [\gl_i - \gl_i(d) \delta_d] & \mathrm{if\ }\gl_i(d) \neq 1 \\
\delta_d & \mathrm{if\ }\gl_i(d) = 1,
\end{cases}
\qquad (i = 0, \dots, d)
$$
where $\delta_d$ is the coordinate row vector $(0, 0, \dots, 0, 1)$, and note that the rows 
of ~$\gLb$, like those of~$\gL$, sum to unity.  Further, define $\Pb := B + R$ to be the sum of the bidiagonal upper triangular matrix~$B$ and the matrix~$R$ with rank at most~$1$ vanishing in all columns except for the last defined by
$$
b_{i j} := 
 \begin{cases}
 \theta_i   & \mbox{if $j = i$} \\
 \frac{1 - \gl_{i + 1}(d)}{1 - \gl_i(d)}(1 - \theta_i)   & \mbox{if $j = i + 1$} \\
 0   & \mbox{otherwise},
 \end{cases} 
$$
and [recalling the notation~\eqref{akdef}]
$$
r_{i d} :=  \left[ 1 - \frac{1 - \gl_{i + 1}(d)}{1 - \gl_i(d)} \right] (1 - \theta_i) 
= \frac{a_{i + 1}}{1 - \gl_i(d)} (1 - \theta_i)
$$
if $\gl_i(d) \neq 1$, and by 
$$
b_{i j} := \delta_{i j}, \qquad r_{i d} := 0
$$ 
if  $\gl_i(d) = 1$.  Observe that the rows of~$\Pb$ sum to unity.  Finally, let $\mb_0$ be the probability row vector
$$
\mb_0 := (1 - m_0(d), 0, \dots, 0, m_0(d)) = (1 - a_0, 0, \dots, 0, a_0).
$$
The following key fact follows by straightforward calculations from \refL{L:gLgen}.

\begl
\label{L:twine}
In the general setting and the above notation,
$$
m_0 = \mb_0 \gLb \qquad\mathrm{and}\qquad \gLb P = \Pb\,\gLb.
$$
\enl

One can use \refL{L:twine} in place of \refL{L:gLgen} to give another proof of \refT{T:gen}.  But much more is possible when the eigenvalues and spectral polynomials of~$P$ are all nonnegative, as we shall henceforth assume.  In that case we have the following conclusions (with all proofs entirely routine): 
\begin{itemize}
\item $\gLb$ and $\Pb$ are both stochastic, and so the semigroups $(P^t)_{t \geq 0}$ and $(\Pb^t)_{t \geq 0}$ are intertwined by the link~$\gLb$.
\item The set~$\Ab$ of absorbing states for a chain~$\Xb$ with kernel~$\Pb$ satisfies
$\Ab = \{\db, \dots, d\}$, where
$$
\db := \min\{i:\gl_i(d) = 1\} = \min\{i:a_i = 0\} - 1 \in \{0, \dots, d\},
$$
and $a_i = 0$ if and only if $i \geq \db + 1$.
\item For a $\Pb$-chain, from each state in $\{0, \dots, \db - 1\}$ the states~$\db$ and~$d$ are each accessible but none of the other states in~$\Ab$ is.
\item The construction of Section~2.4 of~\cite{DF} allows us to build, from~$X$ and independent randomness, a chain~$\Xb$ with initial distribution~$\mb$ and kernel~$\Pb$ such that
$$
\Lc(X_t\,|\,\Xb_0, \dots, \Xb_t) = \gLb(\Xb_t, \cdot) \quad \mbox{for all~$t$}.
$$
\item The time~$T$ to absorption in state~$d$ for~$X$ is the same (sample-pathwise) as the time to absorption (call it~$\Tb$) in~$\Ab$ (i.e.,\ in $\{\db, d\}$).
\item Let~$L$ be the largest value reached by~$\Xb$ prior to absorption, with the convention $L := -1$ if the initial state of~$\Xb$ is~$d$.  Then $\PP(L = k - 1) = a_k$ for all $k = 0, \dots, d$.  Further, conditionally given $L = k - 1$, the amounts of time it takes for~$\Xb$ to move up 
from~$0$ to~$1$, from~$1$ to~$2$, \dots, from $k - 2$ to $k - 1$, and from $k - 1$ to~$\Ab$ are independent geometric random variables with respective success probabilities $1 - \theta_0, 1 - \theta_1, \dots, 1 - \theta_{k - 2}, 1 - \theta_{k - 1}$. 
\end{itemize}

Thus in the case of nonnegative eigenvalues and spectral polynomials we have enriched the conclusion of \refT{T:gen} by means of a stochastic construction identifying (i)~a random variable (namely, $L + 1$) having probability mass function $(a_k)$; and, conditionally given 
$L + 1 = k$, (ii)~individual geometric random variables whose distributions appear in the convolution $\Gc(\theta_0, \dots, \theta_{k - 1})$.

\begr
Recall \refR{R:gen}.  If the conditions described there of detailed balance and complete communication within $\{0, \dots, d - 1\}$ both hold, then $\db = d$ and thus $\Ab = \{d\}$ is a singleton.
\enr
\bigskip

{\bf Acknowledgments.\ }\refS{S:general} was added after the author read the interesting 
paper~\cite{Miclo} by Laurent Miclo; we thank him for providing a preprint, which the results of \refS{S:general} generalize.  We also thank Mark Brown for inspirational discussions.


\begin{thebibliography}{99}
\def\nobibitem#1\par{}

\bibitem{BS}
Brown, M.\ and Shao, Y.~S.\ \ 
Identifying coefficients in the spectral representation for first passage time distributions.
\emph{Probab.\ Eng.\ Inform.\ Sci.} {\bf 1} (1987), 69--74.

\bibitem{DF}
Diaconis, P.\ and Fill, J.~A.\ \ 
Strong stationary times via a new form of duality.
\emph{Ann.\ Probab.} {\bf 18} (1990), 1483--1522.

\bibitem{DM}
Diaconis, P.\ and Miclo, L.\ \ 
On times to quasi-stationarity for birth and death processes.
\emph{J.\ Theoret.\ Probab.} (2009), to appear.

\bibitem{TTS}
Fill, J.~A.\ \ 
Time to stationarity for a continuous-time Markov chain.
\emph{Probab.\ Eng.\ Info.\ Scis.} {\bf 5} (1991), 61--76.

\bibitem{dualityc}
Fill, J.~A.\ \ 
Strong stationary duality for continuous-time Markov chains.  Part~I:\ Theory.
\emph{J.\ Theoret.\ Probab.} {\bf 5} (1992), 45--70.

\bibitem{FillKeil}
Fill, J.~A.\ \ 
The passage time distribution for a birth-and-death chain:\ Strong stationary duality gives a first stochastic proof.
\emph{J.\ Theoret.\ Probab.} (2009), to appear.

\bibitem{HZ}
He, Q-M.\ and Zhang, H.\ \ 
Spectral polynomial algorithms for computing bi-diagonal representations for phase type distributions and matrix-exponential distributions.
\emph{Stoch.\ Models} {\bf 22} (2006), 289--317.

\bibitem{HJ}
Horn, R.~A.\ and Johnson, C.~R.\ \ 
\emph{Matrix Analysis.}
Cambridge University Press, 1985.

\bibitem{KM}
Karlin, S.\ and McGregor, J.\ \ 
Coincidence properties of birth and death processes.
\emph{Pacific J.\ Math.} {\bf 9} (1959), 1109--1140.

\bibitem{MW}
Micchelli, C.~A.\ and Willoughby, R.~A.\ \ 
On functions which preserve the class of Stieltjes matrices.
\emph{Lin.\ Alg.\ Appl.} {\bf 23} (1979), 141--156.

\bibitem{Miclo}
Miclo, L.\ \ 
On absorbtion [sic] times and Dirichlet eigenvalues.
Preprint, 2008.

\end{thebibliography}
\end{document}